\documentclass[12pt]{article}

\usepackage{amsthm}

\usepackage[left=1cm,right=1cm,top=0.5cm,bottom=1.5cm,bindingoffset=0cm]{geometry}

\usepackage[T2A]{fontenc}
\usepackage[utf8]{inputenc}
\usepackage[english]{babel}
\usepackage{amsmath}
\usepackage{amsfonts}
\usepackage{mathtools}
\usepackage{graphicx}
\usepackage{systeme}
\usepackage{blindtext}
\usepackage{amssymb}
\usepackage{hyperref}
\usepackage{csquotes}
\usepackage{bbm}
\usepackage[
backend=biber,
style=numeric,
sorting=none
]{biblatex}
\usepackage[normalem]{ulem}
\graphicspath{ {./} }

\newtheorem{theorem}{Theorem}
\newtheorem{corollary}{Corollary}[section]
\newtheorem{lemma}{Lemma}[section]
\newtheorem*{lemma*}{Lemma}
\newtheorem*{den*}{Notation}

\newtheorem*{remark*}{Remark}

\newtheorem{proposition}[theorem]{Proposition}

\theoremstyle{remark}

\theoremstyle{definition}
\newtheorem{definition}{Definition}[section]

\newcommand{\mP}{\mathbf{P}}
\newcommand{\mQ}{\mathbf{Q}}
\newcommand{\mE}{\mathbf{E}}

\newcommand{\lst}{\mathrm{list}}

\newcommand{\ind}[1]{\mathbbm{1}\left\{#1\right\}}
\newcommand{\skob}[1]{\left(#1\right)}
\newcommand{\fig}[1]{\left\{#1\right\}}

\newcommand{\flr}[1]{\left\lfloor#1\right\rfloor}

\title{List chromatic number of $G(n,r,s)$ graphs}
\date{}
\author{M. M. Koshelev}
\begin{document}
\maketitle

\begin{abstract}
    We study the list chromatic number of generalized Johnson graphs. Recall that the vertices of such graphs are all possible $r$-element subsets of $[n] = \{1, \dotsc, n\}$ and an edge is drawn between two vertices if and only if corresponding sets have exactly $s$ common elements. We prove that for any fixed $s \geq 1$ and $r \geq 2s + 2$ we have $\chi_L(G(n,r,s)) = \Theta(n^{s+1}\ln n)$. In case where $s = 0$ we prove that for all $4\ln n \leq r \leq n/20$ we have $\chi_L(G(n,r,0)) \geq \frac{n}{120}\ln\skob{\frac{en}{r}}$. Together with previous results of Bulankina, Inozemtsev and Kupavskii this proves that $\chi_L(G(n,r,0)) = \Theta\skob{n\ln\skob{\frac{en}{r}}}$ uniformly for $2 \leq r \leq (1/2-\varepsilon)n$.
\end{abstract}

\section{Introduction}

In this paper we study list colorings of the graphs $G(n,r,s)$. For integers $n > r > s \geq 0$, let $[n]=\{1,\ldots,n\}$ and let $\binom{[n]}{r}$ denote the family of all $r$-element subsets of $[n]$. The graph $G(n,r,s)$ has vertex set $\binom{[n]}{r}$, with two distinct vertices $A$ and $B$ adjacent if and only if $|A\cap B|=s$. This family includes the Kneser graphs $KG(n,r)=G(n,r,0)$ and the Johnson graphs $J(n,r)=G(n,r,r-1)$. Independent sets in $G(n,r,s)$ are precisely uniform set systems with one forbidden intersection size, which connects these graphs to extremal set theory~\cite{FrFu}. Their realizations as Euclidean distance graphs also make them useful in combinatorial geometry~\cite{FrW}. Other applications of these graphs include coding and Ramsey theory (see \cite{Codes, Codes2, Spectrum}). A lot of literature studies various properties of these graphs (see, for example, \cite{RaiMeta,MSp,Shu1,Shu2}).

A substantial part of this research concerns coloring parameters.
For example, the classical result of Lov\'asz~\cite{Lovasz1978} determines the chromatic number for Kneser graphs.
\begin{theorem}\label{LovKG}
    Let $n \geq 2r$. Then
    $$
        \chi(G(n,r,0)) = n-2r+2.
    $$
\end{theorem}

For general $r$ and $s$, chromatic numbers have been investigated using methods from extremal set theory, algebra, and design theory; see, for example,~\cite{RaiMetaAlphaChrom,FrW,ZChrom}.
Other coloring parameters have also attracted attention. In particular, the clique-chromatic number is the minimum number of colors in a vertex coloring with no monochromatic inclusion-maximal clique of size at least two. Bounds on this parameter for several families of $G(n,r,s)$ graphs were obtained in~\cite{Rai1,DAM,ClChr2}.

Here we focus on the \emph{list chromatic number}, also known as the \emph{choice number}. List coloring was introduced independently by Vizing~\cite{Vizing1976} and Erd\H{o}s, Rubin, and Taylor~\cite{ERT1980}. A \emph{list assignment} $L$ associates a set $L(v)$ of available colors with each vertex $v$ of a graph $G$. An \emph{$L$-coloring} is a proper vertex coloring $c$ such that $c(v)\in L(v)$ for every $v\in V(G)$. The list chromatic number $\chi_L(G)$ is the least integer $k$ such that $G$ admits an $L$-coloring for every list assignment satisfying $|L(v)|\ge k$ for all $v\in V(G)$. 

Note that the following proposition trivially follows from the definition.

\begin{proposition}\label{trivial}
$$
\chi(G)\le\chi_L(G)\le\Delta(G)+1,
$$
where $\Delta(G)$ denotes the maximum degree of $G$.
\end{proposition}

List chromatic number was studied for many families of graphs including complete multipartite graphs (see \cite{Alon1992, Shabanov2011}), planar graphs (see \cite{Thomassen1994}), line graphs of bipartite graphs (see \cite{Galvin1995}) and pseudorandom graphs (see \cite{SpLChr}).

Unfortunately, little is known about the behaviour of list chromatic number of $G(n,r,s)$ graphs. For positive $s$ we in fact are aware of only one result due to  Zakharov~\cite{ZChrom}.

\begin{theorem}
    \label{Zakharov}
    For fixed integers $r > s \geq 0$ we have
    $$
    \chi_L(G(n,r,s))=O(n^{s+1}\ln n).
    $$
\end{theorem}

Together with Proposition \ref{trivial}, this gives
$$
\chi_L(G(n,r,s))=\Theta(n^{r-s})
\qquad\text{for }r\leq2s+1,
$$
and
$$
c_{r,s}n^{s+1}\leq\chi_L(G(n,r,s))
\leq C_{r,s}n^{s+1}\ln n \qquad\text{for }r > 2s+1.
$$
Here $c_{r,s}$ and $C_{r,s}$ are positive constants.

The Kneser case is a little bit more studied.
Bulankina and Kupavskii~\cite{BulKup} proved the following two theorems.

\begin{theorem}\label{Low1}
    For every fixed $\varepsilon\in(0,1/2)$ and any $2 \leq r \leq n^{1/2-\varepsilon}$ we have
    $$
        \chi_L(G(n,r,0)) \geq C_{\varepsilon}n\ln n
    $$
    for some $C_\varepsilon > 0$.
\end{theorem}

\begin{theorem}\label{Up}
    For any $r$ and $n$ such that $2r \leq n$ we have
    $$
        \chi_L(G(n,r,0))\le n\ln\frac{n}{r}+n.
    $$
\end{theorem}

Inozemtsev and Kupavskii~\cite{InKup} subsequently strengthened the lower bound. To be more precise, they obtained the following theorem.

\begin{theorem}\label{Low2}
    Fix $\varepsilon > 0$ and let $3 \leq r \leq \sqrt\frac{n}{1+\varepsilon}$. Then
    $$
    \chi_L(G(n,r,0)) \geq \skob{\frac{r-2}{8r}-\frac{\varepsilon}{2}}n\ln n
    $$
    for all $n > n_0(\varepsilon)$.
\end{theorem}

In this paper we address the following two questions. The first one is the correct growth rate of $\chi_L(G(n,r,s))$ for $r > 2s+1$. We give an answer to this question in a slightly more restraining regime $r \geq 2s+3$. 

\begin{theorem}
    \label{main_gnrs}
    Fix integers $s > 0$ and $r \geq 2s+3$. Then
    $$
        \chi_L(G(n,r,s)) \geq (1-o(1))\frac{r-2s-1}{8\ln 2\, \binom{r}{s+1}}\binom{n}{s+1}\ln n.
    $$
\end{theorem}

Furthermore we give an upper bound of the same order of magnitude with explicit constant. Note that the result about the correct order of magnitude is not new, since Theorem \ref{Zakharov} already settles this question. Our results, however, have two major advantages. First, these results can be applied to almost any graph and second, the explicit constant in our results is smaller than the implicit constant in Theorem \ref{Zakharov}.

We give two bounds. The second one is stronger but at the same time a bit less general.

Before we go on to the results, let us give a definition of fractional chromatic number.

\begin{definition}
    Let $\mathcal{Q}$ be a set of all possible distributions on the collection of independent sets in graph $G$. Define 
    $$
        \chi_f(G) := \min_{\mathbf{Q} \in \mathcal{Q}} \max_{v \in G} \mathbf{Q}^{-1}(v \in I),
    $$
    where $I$ is a random independent set. We call $\chi_f(G)$ the fractional chromatic number of $G$.
\end{definition}

A lot of properties of fractional chromatic number are known. For example, it is easy to prove that for vertex transitive graph $G$ we have
$$
    \chi_f(G) = \frac{|V_G|}{\alpha(G)}.
$$
For fixed integers $s \geq 0$ and $r > 2s+1$, and for sufficiently large $n$, we have 
$$
\chi_f(G(n,r,s)) = \frac{1}{\binom{r}{s+1}}\binom{n}{s+1}.
$$

A simple probabilistic argument allows us to prove the following theorem.

\begin{theorem}
    \label{UB1}
    For any graph $G$ on $n$ vertices we have 
    $$
        \chi_L(G) \leq k = \lceil\chi_f(G)(\ln n + 1)\rceil.
    $$
\end{theorem}

Note that a similar theorem with slightly better bound in case of fixed $\chi_f(G)$ and slightly worse bound in case of arbitrary $\chi_f(G)$ was proved in~\cite{WCvB} (see Theorems 5 and 30). 

This theorem alone gives us something like 

$$
    \chi_L(G(n,r,s)) \leq (1+o(1))\frac{r}{\binom{r}{s+1}}\binom{n}{s+1}\ln n.
$$

The second result is a bit more involved.

\begin{theorem}
    \label{UB2}
    Let $G$ be a graph on $n \geq 2$ vertices. Let $a(G) = n\chi^{-1}_f(G)$ (note that $a(G) = \alpha(G)$ for vertex transitive graphs). Assume that $\frac{\chi_f(G)}{\sqrt{\ln a(G)}} > 10\ln n$ and $a(G) \geq 3$. Then
    $$
        \chi_L(G) \leq k = \lceil\chi_f(G)\skob{\ln a(G) + \ln \ln a(G)}\rceil + 3\left\lceil\frac{\chi_f(G)}{\sqrt{\ln a(G)}}\right\rceil
    $$
\end{theorem}

For fixed integers $s \geq 0$ and $r > 2s+1$, all conditions of this theorem are satisfied for $G(n,r,s)$ when $n$ is sufficiently large. We can apply it and obtain a slightly better bound.

\begin{corollary}    
    Fix $r$ and $s$ such that $r > 2s+1$ and $s \geq 0$. We have 
    $$
        \chi_L(G(n,r,s)) \leq (1+o(1))\frac{r-s-1}{\binom{r}{s+1}}\binom{n}{s+1}\ln n.
    $$
\end{corollary}

The second question we address is the correct growth rate of $\chi_L(G(n,r,0))$ for $r = \Omega(n^{1/2})$. To this end we prove the following theorem.

\begin{theorem}
    \label{main_kgnr}
	Let $n \geq 1000$, $r \geq 4\ln n$, $20r \leq n$. Then
	$$
		\chi_L(G(n,r,0)) \geq \frac{n}{120}\ln\skob{\frac{en}{r}}.
	$$
\end{theorem}

Note that the combination of Theorems \ref{Low1}, \ref{Up}, \ref{main_kgnr} and \ref{LovKG} together with Proposition \ref{trivial} shows that 

$$ 
    \chi_L(G(n,r,0)) = \Theta\skob{n\ln\skob{en/r}}
$$
uniformly for $2 \leq r \leq (1/2 - \varepsilon)n$.

The rest of the paper is structured as follows. In section 2 we formulate auxiliary lemmas and theorems. In section 3 we prove these auxiliary results. In section 4 we prove the lower bounds for both $r \geq 2s+3$ and $s = 0$ cases. Finally in section 5 we prove upper bounds on $\chi_L(G)$.

\section{Auxiliary results}

In our proofs we need a supersaturation result on the number of edges inside a set of vertices. For the case $r \geq 2s+3$ we use the following result that is a compilation of two theorems from \cite{ShKup}.

\begin{theorem}\label{SuperSat}
    Fix two integers, $r$ and $s$, $s > 0$ and $r > 2s + 2$. For any $\varepsilon > 0$ there exists $\delta = \delta(\varepsilon, r, s) > 0$ such that for all $n \geq n_0(\varepsilon,r,s)$ and any $A \subset G$ such that $|A| > (1 + \varepsilon)\binom{n-s-1}{r-s-1}$ we have $e(A) > \frac{\delta |A|^2}{n^s}$.
\end{theorem}

For the $s = 0$ case we prove the following auxiliary theorem.

\begin{theorem}
	\label{Spectr}
	Fix $r$, $n$ such that $2r \leq n$. Let $A$ be the set of vertices of $G(n,r,0)$ such that $|A| > \alpha(G(n,r,0)) = \binom{n-1}{r-1}$. Then
	$$
		e(A) \geq \frac{\binom{n-r-1}{r-1}}{2\binom{n-1}{r-1}}|A|\skob{|A| - \binom{n-1}{r-1}}.
	$$
\end{theorem}

The other ingredient we need is the result on containers in $G(n,r,s)$. The notion of the container was first introduced by Saxton and Thomason and independently by Balogh, Morris and Samotij (see \cite{Cont1, Cont2}). Before we proceed, let us give the definition. Note that here we a slightly different definition that is more convenient for us.

\begin{definition}
    Consider a graph $G$. Let $\mathcal{C}$ be a collection of subsets of its vertices satisfying the following properties.
    \begin{enumerate}
        \item Each independent subset of $G$ is contained in at least one of the sets $C \in \mathcal{C}$.

        \item For each set $C \in \mathcal{C}$ its size does not exceed $K\alpha(G)$.

        \item The cardinality $|\mathcal{C}|$ does not exceed $f$.
    \end{enumerate}
    We call such a collection $(K, f)$-container.
\end{definition}

For example, $\mathcal{C} = \{G\}$ is a $(|G| / \alpha(G), 1)$-container and the set of all subsets of size $\alpha(G)$ is a $\skob{1, \binom{n}{\alpha(G)}}$-container.

We again need two auxiliary results here, one for the $r \geq 2s+3$ case and the second for $s = 0$ case. 

\begin{lemma}
    \label{mainL}
    Fix $s > 0$, $r \geq 2s+2$ and $\varepsilon > 0$. There exists $\delta = \delta_{\varepsilon, r, s} > 0$ such that there exists a $(1 + \varepsilon, \exp\skob{\delta n^s \ln^2 n})$-container in $G(n,r,s)$.
\end{lemma}

\begin{lemma}
    	\label{mainLKG}
    	Fix $r$, $n$ such that $r \geq 10$, $n \geq 1000$, $3r \leq n$ and take any $1 \geq \varepsilon > \frac{1}{n}$. Then there exists a $\skob{1 + \varepsilon,2\skob{\frac{e\binom{n}{r}}{k}}^k}$-container in $G(n,r,0)$, where 
	$$
		k =  \flr{\frac{2(2+\varepsilon)\binom{n-1}{r-1}\ln \binom{n}{r}}{\varepsilon\binom{n-r-1}{r-1}}}.
	$$

    For instance, there exists a $\skob{2, 2\skob{\frac{e\binom{n}{r}}{k}}^k}$-container in $G(n,r,0)$, where 
	$$
		k =  \flr{\frac{6\binom{n-1}{r-1}\ln \binom{n}{r}}{\binom{n-r-1}{r-1}}}.
	$$
\end{lemma}

\section{Proof of auxiliary results}

\subsection{Proof of Theorem \ref{Spectr}}

\begin{proof}
	Recall that $G(n,r,0)$ is an $(N, d, \lambda)$ graph with $N = \binom{n}{r}$, $d = \binom{n-r}{r}$ and $\lambda = \binom{n-r-1}{r-1}$. Note that $d = \frac{n-r}{r}\lambda$, $N = \frac{n}{r}\alpha$.

	We thus have
	$$
		2e(A) \geq \frac{d}{N}|A|^2 - \frac{\lambda}{N}|A|(N-|A|) = \frac{d+\lambda}{N}|A|^2 - \lambda|A| = \frac{\lambda}{\alpha}\skob{\frac{\frac{n-r}{r}+1}{\frac{n}{r}}|A|^2-|A|\alpha} = \frac{\lambda}{\alpha}|A|(|A| - \alpha),
	$$
	which is exactly what we need.
\end{proof}

\subsection{Proof of Lemma \ref{mainL}}

\begin{proof}[Proof of the lemma]
    Let us consider an arbitrary independent set $I$ in $G(n,r,s)$ and perform the following algorithm. Put $V_0 = V_{G(n,r,s)}$ and $U_0 = \varnothing$. At $i$-th step select the vertex $v_i$ with maximum degree in $G(n,r,s)[V_{i-1}]$. If $v_i \in I$, put $V_i = V_{i-1} \setminus (\{v_i\} \cup N(v_i))$ and $U_i = U_{i-1} \cup \{v_i\}$. If $v_i \notin I$, we put $V_i = V_{i-1} \setminus \{v_i\}$ and $U_i = U_{i-1}$. We stop whenever $|V_i| \leq (1+\varepsilon/2)\binom{n-s-1}{r-s-1}$. In the remaining part of the proof index $i$ denotes the last step of the algorithm (that is, $|V_i| \leq (1+\varepsilon/2)\binom{n-s-1}{r-s-1}$).

    Let us obtain some properties of $V_i$ and $U_i$. First of all, let us notice that $V_i$ is a function of $U_i$. Indeed, if two independent sets $I_1$ and $I_2$ are the same on $V_{G(n,r,s)} \setminus V_i$, the algorithm does just the same steps for them and thus the final sets will be the same. This means that the number of different $V_i \cup U_i$ over all $I$ is exactly the same as the number of different $U_i$. 

    The second thing to notice is that whenever we add a vertex to $U_{j+1}$, we throw out $\Delta(G(n,r,s)[V_j])+1$ vertices from $V_j$. This can be trivially bounded from below with the use of Theorem \ref{SuperSat}: 

    $$
        \Delta(G(n,r,s)[V_j]) \geq \frac{2e(G(n,r,s)[V_j])}{|V_j|} \geq \frac{2\delta|V_j|}{n^s}.
    $$

    We thus have 
    $$
    |V_{j+1}| \leq \skob{1-\frac{C}{n^s}}|V_j| \Rightarrow 1 \leq |V_{i-1}| \leq \skob{1-\frac{C}{n^s}}^{|U_i|-1}\binom{n}{r} \Rightarrow |U_i| \leq \Tilde{C}n^s\ln n
    $$
    for some $\Tilde{C} = \Tilde{C}(r,s,\varepsilon)$.

    Finally, the last thing to notice is that $V_i \cup U_i$ contains $I$.

    Using these three properties we can see that the collection of all possible $V_i \cup U_i$ is indeed a container, each set $V_i \cup U_i$ is of size at most $(1 + \varepsilon/2 + o_n(1))\binom{n-s-1}{r-s-1} \leq (1 + \varepsilon)\binom{n-s-1}{r-s-1}$ and there are at most 
    $$
        \sum_{j = 0}^{\flr{\Tilde{C}n^s\ln n}} \binom{\binom{n}{r}}{j} = \exp\skob{O(n^s\ln^2 n)}
    $$
    such sets, which is exactly what we need.
\end{proof}

\subsection{Proof of Lemma \ref{mainLKG}}

\begin{proof}
	Let us consider an arbitrary independent set $I$ in $G(n,r,0)$ and perform the following algorithm. Put $V_0 = V_{G(n,r,0)}$ and $U_0 = \varnothing$. At $i$-th step select the vertex $v_i$ with maximum degree in $G(n,r,0)[V_{i-1}]$. If $v_i \in I$, put $V_i = V_{i-1} \setminus (\{v_i\} \cup N(v_i))$ and $U_i = U_{i-1} \cup \{v_i\}$. If $v_i \notin I$, we put $V_i = V_{i-1} \setminus \{v_i\}$ and $U_i = U_{i-1}$. We stop whenever $|V_i| \leq (1+\varepsilon/2)\binom{n-1}{r-1}$. In the remaining part of the proof index $i$ denotes the last step of the algorithm (that is, $|V_i| \leq (1+\varepsilon/2)\binom{n-1}{r-1}$).
	
	Let us obtain some properties of $V_i$ and $U_i$. First of all, let us notice that $V_i$ is a function of $U_i$. Indeed, if two independent sets $I_1$ and $I_2$ are the same on $V_{G(n,r,0)} \setminus V_i$, the algorithm does just the same steps for them and thus the final sets will be the same. This means that the number of different $V_i \cup U_i$ over all $I$ is exactly the same as the number of different $U_i$. 
	
	The second thing to notice is that whenever we add a vertex to $U_{j+1}$, we throw out $\Delta(G(n,r,0)[V_j])+1$ vertices from $V_j$. This can be trivially bounded from below with the use of Theorem \ref{Spectr}: 
	
	$$
		\Delta(G(n,r,0)[V_j]) \geq \frac{2e(G(n,r,0)[V_j])}{|V_j|} \geq \binom{n-r-1}{r-1}\frac{|V_j| - \binom{n-1}{r-1}}{\binom{n-1}{r-1}} \geq \frac{\varepsilon / 2}{1 + \varepsilon / 2}\frac{\binom{n-r-1}{r-1}}{\binom{n-1}{r-1}}|V_j|.
	$$
	
	We thus have 
	\begin{multline*}
		|V_{j+1}| \leq \skob{1-\frac{\varepsilon}{2 + \varepsilon}\frac{\binom{n-r-1}{r-1}}{\binom{n-1}{r-1}}}|V_j| \Rightarrow \\ 
		\Rightarrow 1 \leq |V_{i-1}| \leq \skob{1-\frac{\varepsilon}{2 + \varepsilon}\frac{\binom{n-r-1}{r-1}}{\binom{n-1}{r-1}}}^{|U_i|-1}\binom{n}{r} \leq \\ 
		\leq  \exp\fig{-\frac{\varepsilon}{2 + \varepsilon}\frac{\binom{n-r-1}{r-1}}{\binom{n-1}{r-1}}(|U_i|-1)}\binom{n}{r} \Rightarrow \\
		\Rightarrow |U_i| \leq 1 + \frac{(2+\varepsilon)\binom{n-1}{r-1}\ln \binom{n}{r}}{\varepsilon\binom{n-r-1}{r-1}} \leq \frac{2(2+\varepsilon)\binom{n-1}{r-1}\ln \binom{n}{r}}{\varepsilon\binom{n-r-1}{r-1}}.
	\end{multline*}
	    
	Note that since $n \geq 3r$, we have 
	$$
		\frac{2(2+\varepsilon)\binom{n-1}{r-1}\ln \binom{n}{r}}{\varepsilon\binom{n-r-1}{r-1}} \leq  2\frac{3^{10} 9!\binom{n-1}{r-1}}{\varepsilon n^8} \leq 2\frac{3^{10}9!}{n^7}\binom{n-1}{r-1} \leq \frac{100}{n^2}\binom{n-1}{r-1} \leq \frac{1}{10n}\binom{n-1}{r-1}
	$$
	and thus $|U_i \cup V_i| \leq \skob{1 + \varepsilon/2 + \frac{1}{10n}}\binom{n-1}{r-1} \leq (1+\varepsilon)\binom{n-1}{r-1}$.
	
	Finally, the last thing to notice is that $V_i \cup U_i$ contains $I$.

	Now we put all of this together to obtain that the set of all possible $U_i \cup V_i$ is indeed a container, each $V_i \cup U_i$ has size at most $(1+\varepsilon)\binom{n-1}{r-1}$ and the number of distinct $V_i \cup U_i$ is at most
	$$
		S = \sum_{i = 0}^{k}\binom{\binom{n}{r}}{i}.
	$$
	Since $k \leq \binom{n-1}{r-1} = \frac{r}{n}\binom{n}{r} \leq \frac{1}{3}\binom{n}{r}$, we have 
	$$
		\binom{\binom{n}{r}}{i} \binom{\binom{n}{r}}{i+1}^{-1} = \frac{(i+1)}{\binom{n}{r}-i} \leq \frac{1}{2}
	$$
	for $i < k$ and thus
	$$
		S \leq 2\binom{\binom{n}{r}}{k} \leq 2\skob{\frac{\binom{n}{r}e}{k}}^{k}.
	$$ 
\end{proof}

\section{Lower bounds}

\subsection{Proof of Theorem \ref{main_gnrs}}

\begin{proof}
    Fix any $0 < \varepsilon < 1$ and let us take $\delta = (1-\varepsilon)\frac{r-2s-1}{8\ln 2\binom{r}{s+1}}$. The main idea is the following. We take $2q = 2\flr{\delta \binom{n}{s+1}\ln n}$ colors and randomly assign each vertex a list containing half of these colors. Assume that we managed to color our graph properly. Let $B_i$ be the set of vertices of color $i$. Then $B_i$ is contained in some set from (any) container. This means that we can assign each color $c$ a set $\varphi(c)$ from the container such that for each $i$ and each $v \in B_i$ we have at least one $C$ from the container such that $v \in C$ and $\varphi(i) = C$. It remains to prove that with positive probability for any $\varphi$ there exists at least one such $v$ that $v \notin \bigcup_{j = 1}^{q} \varphi(c_j)$, where $\{c_1, \dotsc, c_q\}$ is the list of colors assigned to $v$. We call $\varphi$ with this condition bad and all other $\varphi$ good.

    Let us now give a rigorous proof. Consider the $(2, \exp\{Cn^{s}\ln^2n\})$-container $\mathcal{C}$ from Lemma \ref{mainL}. Let us fix some function $\varphi$ that maps each color to some set $C \in \mathcal{C}$. Let us bound the probability that $\varphi$ is good. Let $\mathcal{C}(v)$ be the collection of all sets in $\mathcal{C}$ containing $v$. Let also $\lst(v)$ be the list of colors for $v$. Finally, put $\xi_v = \ind{\lst(v) \cap \varphi^{-1}(\mathcal{C}(v)) = \varnothing}$. Then $\varphi$ is good if and only if $\sum_{i =1}^{\binom{n}{r}}\xi_i = 0$. Since all $\xi_i$ are independent, we have
    $$
        \mP\skob{\sum_{i = 1}^{\binom{n}{r}}\xi_i = 0} = \prod_{i = 1}^{\binom{n}{r}}\mP(\xi_i = 0).
    $$

    Let us now bound $\mP(\xi_v = 0)$. In fact we will not bound all such probabilities, but only the ones with small $\varphi^{-1}(\mathcal{C}(v))$. Note that since each $C \in \mathcal{C}$ is of size at most $2\binom{n-s-1}{r-s-1}$, we have that 
    $$
    \sum_{v = 1}^{\binom{n}{r}}|\varphi^{-1}(\mathcal{C}(v))| = \sum_{c = 1}^{2q}|\varphi(c)| \leq 4q\binom{n-s-1}{r-s-1}.
    $$
    We thus obtain that there are at most $\frac{1}{2}\binom{n}{r}$ vertices such that 
    $$
    |\varphi^{-1}(\mathcal{C}(v))| \geq \frac{8q\binom{n-s-1}{r-s-1}}{\binom{n}{r}} \leq 8\delta\binom{r}{s+1}\ln n.
    $$

    Now consider any $v$ such that $|\varphi^{-1}(\mathcal{C}(v))| \leq 8\delta\binom{r}{s+1}\ln n$. We can bound $\mP(\xi_v = 0)$ for big enough $n$ as follows.
    \begin{multline*}
        \mP(\xi_v = 0) = 1- \mP(\xi_v = 1) = 1 -\frac{\binom{2q-|\varphi^{-1}(\mathcal{C}(v))|}{q}}{\binom{2q}{q}} = 1 -\prod_{i = 0}^{q-1}\frac{2q-|\varphi^{-1}(\mathcal{C}(v))| - i}{2q-i} \leq \\ \leq 1 - \skob{\frac{q -8\delta\binom{r}{s+1}\ln n}{2q}}^{8\delta\binom{r}{s+1}\ln n} \leq 1 - 2^{-(1+\varepsilon)8\delta\binom{r}{s+1}\ln n} = 1 - n^{-(1+\varepsilon)8\delta\binom{r}{s+1}\ln 2}. 
    \end{multline*}

    We thus have
    $$
         \mP\skob{\sum_{i = 1}^{\binom{n}{r}}\xi_i = 0} \leq \skob{1 - n^{-(1+\varepsilon)8\delta\binom{r}{s+1}\ln 2}}^{\frac{1}{2}\binom{n}{r}} \leq \exp\left\{-\frac{1}{2}\binom{n}{r}n^{-(1+\varepsilon)8\delta\binom{r}{s+1}\ln 2}\right\}.
    $$
    It remains to notice that taking union bound over all $\varphi$ yields 
    \begin{multline*}
        \mP(\text{exists at least } 1 \text{ good } \varphi) \leq (\exp\{Cn^s\ln^2 n\})^{2\delta n^{s+1}\ln n}\exp\left\{-\frac{1}{2}\binom{n}{r}n^{-(1+\varepsilon)8\delta\binom{r}{s+1}\ln 2}\right\} = \\ = \exp\left\{2\delta Cn^{2s+1}\ln^3 n - \frac{1}{2}\binom{n}{r}n^{-(1+\varepsilon)8\delta\binom{r}{s+1}\ln 2}\right\} = \\ = \exp\{2\delta Cn^{2s+1}\ln^3 n - \Omega(n^{r-(1-\varepsilon^2)(r-2s-1)})\} = o(1).
    \end{multline*}
    It remains to notice that we are done since $\varepsilon$ can be made arbitrarily small.
\end{proof}

\subsection{Proof of Theorem \ref{main_kgnr}}

\begin{proof}
    	Let us take $\delta = \frac{1}{120}$. The main idea is the following. We take $2q = 2\flr{\delta n \ln\skob{\frac{en}{r}}}$ colors and randomly assign each vertex a list containing half of these colors. Assume that we managed to color our graph properly. Let $B_i$ be the set of vertices of color $i$. Then $B_i$ is contained in some set from (any) container. This means that we can assign each color $c$ a set $\varphi(c)$ from the container such that for each $i$ and each $v \in B_i$ we have at least one $C$ from the container such that $v \in C$ and $\varphi(i) = C$. It remains to prove that with positive probability for any $\varphi$ there exists at least one such $v$ that $v \notin \bigcup_{j = 1}^{q} \varphi(c_j)$, where $\{c_1, \dotsc, c_q\}$ is the list of colors assigned to $v$. We call $\varphi$ with this condition bad and all other $\varphi$ good.

    	Let us now give a rigorous proof. Consider the container $\mathcal{C}$ from Lemma \ref{mainLKG}. Let us fix some function $\varphi$ that maps each color to some set $C \in \mathcal{C}$. Let us bound the probability that $\varphi$ is good. Let $\mathcal{C}(v)$ be the collection of all sets in $\mathcal{C}$ containing $v$. Let also $\lst(v)$ be the list of colors for $v$. Finally, put $\xi_v = \ind{\lst(v) \cap \varphi^{-1}(\mathcal{C}(v)) = \varnothing}$. Then $\varphi$ is good if and only if $\sum_{i =1}^{\binom{n}{r}}\xi_i = 0$. Since all $\xi_i$ are independent, we have
    	$$
		\mP\skob{\sum_{i = 1}^{\binom{n}{r}}\xi_i = 0} = \prod_{i = 1}^{\binom{n}{r}}\mP(\xi_i = 0).
    	$$

    	Let us now bound $\mP(\xi_v = 0)$. In fact we will not bound all such probabilities, but only the ones with small $\varphi^{-1}(\mathcal{C}(v))$. Note that since each $C \in \mathcal{C}$ is of size at most $2\binom{n-1}{r-1}$, we have that 
    	$$
    		\sum_{v = 1}^{\binom{n}{r}}|\varphi^{-1}(\mathcal{C}(v))| = \sum_{c = 1}^{2q}|\varphi(c)| \leq 4q\binom{n-1}{r-1}.
    	$$
    	We thus obtain that there are at most $\frac{1}{2}\binom{n}{r}$ vertices such that 
    	$$
    		|\varphi^{-1}(\mathcal{C}(v))| \geq \frac{8q\binom{n-1}{r-1}}{\binom{n}{r}} \leq 8\delta r \ln\skob{\frac{en}{r}}.
    	$$

	Now consider any $v$ such that $|\varphi^{-1}(\mathcal{C}(v))| \leq  \frac{8qr}{n}$. We can bound $\mP(\xi_v = 0)$ for big enough $n$ as follows.
    	\begin{multline*}
		\mP(\xi_v = 0) = 1- \mP(\xi_v = 1) = 1 -\frac{\binom{2q-|\varphi^{-1}(\mathcal{C}(v))|}{q}}{\binom{2q}{q}} = 1 -\prod_{i = 0}^{q-1}\frac{2q-|\varphi^{-1}(\mathcal{C}(v))| - i}{2q-i} \leq \\ \leq 1 - \skob{\frac{q - q/2}{2q}}^{ 8\delta r \ln\skob{\frac{en}{r}}} \leq 1 - 4^{-8\delta r \ln\skob{\frac{en}{r}}} \leq 1 - n^{-12\delta\frac{r\ln\skob{\frac{en}{r}}}{\ln n}}. 
    	\end{multline*}

    	We thus have
	$$
         		\mP\skob{\sum_{i = 1}^{\binom{n}{r}}\xi_i = 0} \leq \skob{1 - n^{-12\delta\frac{ r\ln\skob{\frac{en}{r}}}{\ln n}}}^{\frac{1}{2}\binom{n}{r}} \leq \exp\left\{-\frac{1}{2}\binom{n}{r}n^{-12\delta\frac{r\ln\skob{\frac{en}{r}}}{\ln n}}\right\}.
	$$
	It remains to notice that taking union bound over all $\varphi$ yields 
	\begin{multline*}
		\mP(\text{exists at least } 1 \text{ good } \varphi) \leq 2^{2q}\skob{\frac{e\binom{n}{r}}{k}}^{2kq}\exp\left\{-\frac{1}{2}\binom{n}{r}n^{-12\delta\frac{r\ln\skob{\frac{en}{r}}}{\ln n}}\right\} = \\ 
		= \exp\skob{2q\ln 2 + 2kq + 2kq\ln\binom{n}{r} - 2kq\ln k - \frac12\binom{n}{r}n^{-12\delta\frac{r\ln\skob{\frac{en}{r}}}{\ln n}}} \leq \\
		\leq \exp\skob{4kqr\ln\skob{\frac{en}{r}} - 2kq\ln k - \frac12\binom{n}{r}n^{-12\delta\frac{r\ln\skob{\frac{en}{r}}}{\ln n}}} \leq \\ 
		\leq \exp\skob{4kqr\ln\skob{\frac{en}{r}} - \frac12\binom{n}{r}n^{-12\delta\frac{r\ln\skob{\frac{en}{r}}}{\ln n}}} \leq \\
		\leq  \exp\skob{25\delta rn\frac{\binom{n-1}{r-1}\ln\skob{\frac{en}{r}}\ln \binom{n}{r}}{\binom{n-r-1}{r-1}}\ln\skob{\frac{en}{r}} - \frac12\binom{n}{r}n^{-12\delta\frac{r\ln\skob{\frac{en}{r}}}{\ln n}}} \leq \\
		\leq  \exp\skob{\binom{n}{r}\skob{25\delta \frac{r^3}{\binom{n-r-1}{r-1}}\ln^3\skob{\frac{en}{r}} - \frac12n^{-12\delta\frac{r\ln\skob{\frac{en}{r}}}{\ln n}}}} \leq \\
		\leq  \exp\skob{\binom{n}{r}\skob{25\delta \frac{r!r^2}{(n-2r)^{r-1}}\ln^3\skob{\frac{en}{r}} - \frac12n^{-12\delta\frac{r\ln\skob{\frac{en}{r}}}{\ln n}}}} \leq \\ 
		\leq  \exp\skob{\binom{n}{r}\skob{\frac{r^{r+3}}{n^{r-1}}\ln^3\skob{\frac{en}{r}} - \frac12n^{-12\delta\frac{r\ln\skob{\frac{en}{r}}}{\ln n}}}}.
    \end{multline*}
   
	That is, if inequality 
	$$
		-r+1+(r+3)\frac{\ln r}{\ln n} \leq 4 - r\skob{1-\frac{\ln r}{\ln n}} < -12\delta\frac{r\ln\skob{\frac{en}{r}}}{\ln n},
	$$
	holds we are done. A sufficient condition is
	$$
		r\skob{1-\frac{\ln r}{\ln n}-12\delta\frac{\ln\skob{\frac{en}{r}}}{\ln n}} = r\skob{1-12\delta-\frac{12\delta}{\ln n} - (1-12\delta)\frac{\ln r}{\ln n}} = (1-12\delta)r\skob{1 - \frac{\ln r}{\ln n}} - \frac{12\delta r}{\ln n} > 4.
	$$
	Using the fact that $4\ln n \leq r \leq n/20$ and $\delta = \frac{1}{120}$, we can see that
	$$
 		 (1-12\delta)r\skob{1 - \frac{\ln r}{\ln n}} - \frac{12\delta r}{\ln n} \geq 0.9r\frac{\ln 20}{\ln n} - \frac{0.1}{\ln n}r \geq \frac{\ln 20}{2\ln n}r \geq 2\ln 20 > 4
	$$ 
	and thus we are done.
\end{proof}

\section{Upper bounds}

\subsection{Proof of Theorem \ref{UB1}}

\begin{proof}
    Let $\mathbf{Q}$ be the distribution on independent sets that minimizes $\max_{v} \mQ^{-1}(v \in I)$. For each color $c$ let us assign it to a random set $I_c$ independently from everything else according to distribution $\mQ$. Now for any vertex $v$ and color $c$ we say that $c$ is good for $v$ if $v \in I_c$ and $c$ is in $\lst(v)$. Now we need to bound the probability that for each $v$ we have at least one good color. 
    
    Note that since $\mQ$ is taken to minimize $\max_{v} \mQ^{-1}(v \in I)$ and $I_c$ has distribution $\mQ$ we have $$\max_{v} \mQ^{-1}(v \in I_c) = \chi_f(G)$$ and thus
    $$
        \mQ(v \in I_c) \geq \frac{1}{\chi_f(G)}
    $$
    for any $v \in G$ and any color $c \in \lst(v)$. Now the probability that there is no good color for $v$ can be bounded from above as $(1 - \chi_f^{-1}(G))^{k} \leq \exp\{-k\chi_f^{-1}(G)\} \leq (en)^{-1}$ and it remains to apply the union bound over all $v \in G$.
\end{proof}

\subsection{Proof of Theorem \ref{UB2}}

\begin{proof}
    Fix some lists of size $k$ for all vertices.
    
    Take $l = \left\lceil\frac{\chi_f(G)}{\sqrt{\ln a(G)}}\right\rceil$ and $p = \frac{2l}{k}$. For each color in at least one of the lists we say that it is a reserve color with probability $p$ and main color otherwise. We do this independently for all colors. Let us now notice that for each vertex $v$ we have $\xi_v \sim Bin(k,p)$ reserve colors. 

    It is easy to prove that there exists a coloring such that $\xi_v \in [l, 3l]$ for all $v$ simultaneously. Indeed, we have $\mE \xi_v = 2l$ and thus due to Chernoff's bound
    $$
       1-\mP(\xi_v \in [l, 3l]) = \mP\skob{|\xi_v - \mE\xi_v| > \frac{1}{2}\mE\xi_v} \leq 2e^{-l/5} < 2n^{-2}.
    $$
    Taking union bound over all $v$ gives the desired result.

    So we split all the colors into main and reserve. For now, drop all reserve colors from the lists, so only main colors are left. Note that the size of each list is still at least 
    $$
        k - 3l = \lceil\chi_f(G)\skob{\ln a(G) + \ln \ln a(G)}\rceil.
    $$

    Now let us perform the proof of Theorem \ref{UB1} with slight modifications. Let $\mathbf{Q}$ be the distribution on independent sets that minimizes $\max_{v} \mQ^{-1}(v \in I)$. For each color $c$ let us assign it to a random set $I_c$ independently from everything else according to distribution $\mQ$. Now for any vertex $v$ and color $c$ we say that $c$ is good for $v$ if $v \in I_c$ and $c$ is in $\lst(v)$. The next thing to do is to bound the probability that for a vertex $v$ we have at least one good color $c$
    
    Note that since $\mQ$ is taken to minimize $\max_{v} \mQ^{-1}(v \in I)$ and $I_c$ has distribution $\mQ$ we have $$\max_{v} \mQ^{-1}(v \in I_c) = \chi_f(G)$$ and thus
    $$
        \mQ(v \in I_c) \geq \frac{1}{\chi_f(G)}
    $$
    for any $v \in G$ and any color $c \in \lst(v)$. Now the probability that there is no good color for $v$ can be bounded from above as 
    $$
    (1 - \chi_f^{-1}(G))^{k-3l} \leq \exp\{-(k-3l)\chi_f^{-1}(G)\} \leq \exp\{-\ln a(G) - \ln \ln a(G)\} \leq \frac{1}{a(G)\ln a(G)}.
    $$
    We can now obtain the bound on the expected number of vertices without good color. If $X$ is this number, then $\mE X \leq \frac{n}{a(G)\ln a(G)}$. We can now use Markov inequality to obtain
    $$
        \mQ(X > l) \leq \frac{\mE X}{l} \leq \frac{1}{\sqrt{\ln a(G)}} \leq \frac{1}{\sqrt{\ln 3}} < 1.
    $$
    We can thus choose the sets $I_c$ in such a way that at most $l$ vertices will remain uncolored. But we can then easily color them using the reserve colors (since each list contains at least $l$ of those). 
\end{proof}

\printbibliography

@article{ShKup,
    title = {On supersaturation for Erd\H os-S\'{o}s problem},
    author = {A. Kupavskii and Ya. Shubin},
    %journal = {arxiv},
    volume = {arxiv:2602.10292},
    year = {2026}
}

@article{Shu1,
  title     = "On the Minimal Number of Edges in Induced Subgraphs of Special Distance Graphs",
  author    = "Shubin, Ya.",
  journal   = "Math. Notes",
  volume    =  111,
  number    =  6,
  pages     = "961--969",
  year      =  2022
}

@article{Shu2,
  author  = {N. A. Dubinin and E. A. Neustroeva and A. M. Raigorodskii and Ya. K. Shubin},
  title   = {The Minimum Number of Cliques in Induced Subgraphs of Johnson Graphs},
  journal = {Mathematical Notes},
  volume  = {118},
  number  = {1},
  pages   = {80--94},
  year    = {2025}
  %doi     = {10.1134/S0001434625603454}
}

@article{FrFu,
 author = {P. Frankl and Z. F\"uredi},
 title = "Forbidding just one intersection",
 journal = "J. Combin. Theory Ser. A",
 volume = "39",
 pages = "160--176",
 year = 1985
}

@book{Codes,
    title = {The theory of error-correcting codes},
    year = {1977},
    author = {F. J. MacWilliams and N. J. A. Sloane},
    publisher = {North-Holland},
    note = {}
}

@article{FrW,
    title = {Intersection theorems with geometric consequences},
    year = {1981},
    journal = {Combinatorica},
    author = {P. Frankl and R. Wilson},
    volume = {1},
    pages = {357--368},
    note = {}
}

@article{RaiMeta,
    title = {Cliques and cycles in distance graphs and graphs of diameters},
    year = {2014},
    journal = {Discrete Geometry and Algebraic Combinatorics, AMS, Contemporary Mathematics}, 
    author = {A. M. Raigorodskii},
    volume = {625},
    pages = {93--109},
    note = {}
}

@article{Spectrum,
    title = {An algebraic approach to the association schemes of coding theory},
    year = {1973},
    journal = {Philips Res. Repts. Suppl.},
    author = {P. Delsart},
    volume = {10},
    note = {}
}

@article{Rai1,
    title = {Clique Chromatic Numbers of Intersection Graphs},
    year = {2019},
    journal = {Math. Notes},
    author = {D. A. Zakharov and A. M. Raigorodskii},
    volume = {105},
    pages = {137--139}
}

@article{ZChrom,
title = {Chromatic numbers of Kneser-type graphs},
journal = {J. Combin. Theory Ser. A},
volume = {172},
pages = {105188},
year = {2020},
issn = {0097-3165},
author = {D. Zakharov},
}

@article{Cont1,
  author  = {Saxton, David and Thomason, Andrew},
  title   = {Hypergraph containers},
  journal = {Inventiones Mathematicae},
  year    = {2015},
  volume  = {201},
  number  = {3},
  pages   = {925--992},
  doi     = {10.1007/s00222-014-0562-8},
  url     = {https://doi.org/10.1007/s00222-014-0562-8}
}

@article{Cont2,
  author  = {Balogh, J{\'o}zsef and Morris, Robert and Samotij, Wojciech},
  title   = {Independent sets in hypergraphs},
  journal = {Journal of the American Mathematical Society},
  year    = {2015},
  volume  = {28},
  number  = {3},
  pages   = {669--709},
  doi     = {10.1090/S0894-0347-2014-00816-X},
  url     = {https://doi.org/10.1090/S0894-0347-2014-00816-X}
}

@article{InKup,
title = {Two questions on Kneser colorings},
journal = {Discrete Mathematics},
volume = {349},
number = {2},
pages = {114842},
year = {2026},
issn = {0012-365X},
doi = {https://doi.org/10.1016/j.disc.2025.114842},
url = {https://www.sciencedirect.com/science/article/pii/S0012365X25004509},
author = {Eduard Inozemtsev and Andrey Kupavskii},
}

@article{BulKup,
title = {Choice number of Kneser graphs},
journal = {Discrete Mathematics},
volume = {345},
number = {11},
pages = {113097},
year = {2022},
issn = {0012-365X},
doi = {https://doi.org/10.1016/j.disc.2022.113097},
url = {https://www.sciencedirect.com/science/article/pii/S0012365X2200303X},
author = {Vera Bulankina and Andrey Kupavskii}
}

@article{SpLChr,
	title = {List Coloring of Random and Pseudo-Random Graphs},
	year = {1999},
	author = {Noga Alon and Michael Krivelevich and Benny Sudakov },
	journal = {Combinatorica},
	volume = {19},
	pages = {453--472}
}

@article{RaiMetaAlphaChrom,
    author = {A. M. Raigorodskii and P. S. Sinelnikov-Murylev},
    title = {Johnson graphs, their random subgraphs, and some of their extremal characteristics},
    journal = {Russian Math. Surveys},
    year = {2025},
    volume = {80:3},
	pages = {471--530}
}

@article{Codes2,
	author = {Etzion, Tuvi},
	title = {On the Nonexistence of Perfect Codes in the Johnson Scheme},
	journal = {SIAM Journal on Discrete Mathematics},
	volume = {9},
	number = {2},
	pages = {201-209},
	year = {1996}
}

@ARTICLE{MSp,
  author = {M. Koshelev},
  title = {Spectrum of Johnson graphs},
  journal = {Discrete Mathematics},
  year = {2023},
  volume = {346},
  number = {3},
  pages  = {113262},
  authorscopus = {true},
  authorwos = {true},
  authorvak = {true},
}

@ARTICLE{ClChr2,
  author = {M. Koshelev},
  title = {Lower bounds on the clique-chromatic numbers of some distance graphs},
  journal = {Moscow Journal of Combinatorics and Number Theory},
  year = {2021},
  volume = {10},
  number = {2},
  pages  = {141--148},
  language = {english},
  authorscopus = {true},
  authorvak = {true},
}

@ARTICLE{DAM,
  author = {A. Raigorodskii and M. Koshelev},
  title = {New bounds on clique-chromatic numbers of Johnson graphs},
  journal = {Discrete Applied Mathematics},
  year = {2020},
  volume = {283},
  pages  = {724--729},
  language = {english},
  authorscopus = {true},
  authorwos = {true},
  authorvak = {true},
}

@article{Lovasz1978,
  author  = {L{\'a}szl{\'o} Lov{\'a}sz},
  title   = {{Kneser}'s conjecture, chromatic number, and homotopy},
  journal = {Journal of Combinatorial Theory, Series A},
  volume  = {25},
  number  = {3},
  pages   = {319--324},
  year    = {1978},
  doi     = {10.1016/0097-3165(78)90022-5},
  url     = {https://doi.org/10.1016/0097-3165(78)90022-5}
}

@article{Vizing1976,
  author  = {V. G. Vizing},
  title   = {Vertex colorings with given colors},
  journal = {Diskret. Analiz},
  volume  = {29},
  pages   = {3--10},
  year    = {1976},
  note    = {In Russian}
}

@inproceedings{ERT1980,
  author    = {Paul Erd{\H{o}}s and Arthur L. Rubin and Herbert Taylor},
  title     = {Choosability in graphs},
  booktitle = {Proceedings of the West Coast Conference on Combinatorics,
               Graph Theory and Computing
               (Humboldt State University, Arcata, California, 1979)},
  series    = {Congressus Numerantium},
  volume    = {26},
  pages     = {125--157},
  publisher = {Utilitas Mathematica Publishing},
  address   = {Winnipeg},
  year      = {1980},
  url       = {https://www.renyi.hu/~p_erdos/1980-07.pdf}
}

@article{Alon1992,
  author  = {Alon, Noga},
  title   = {Choice numbers of graphs: a probabilistic approach},
  journal = {Combinatorics, Probability and Computing},
  year    = {1992},
  volume  = {1},
  number  = {2},
  pages   = {107--114},
  doi     = {10.1017/S0963548300000122}
}

@article{Thomassen1994,
  author  = {Thomassen, Carsten},
  title   = {Every planar graph is 5-choosable},
  journal = {Journal of Combinatorial Theory, Series B},
  year    = {1994},
  volume  = {62},
  number  = {1},
  pages   = {180--181},
  doi     = {10.1006/jctb.1994.1062}
}

@article{Galvin1995,
  author  = {Galvin, Fred},
  title   = {The list chromatic index of a bipartite multigraph},
  journal = {Journal of Combinatorial Theory, Series B},
  year    = {1995},
  volume  = {63},
  number  = {1},
  pages   = {153--158},
  doi     = {10.1006/jctb.1995.1011}
}

@article{Shabanov2011,
  author  = {Shabanov, Dmitry A.},
  title   = {On a generalization of {Rubin}'s theorem},
  journal = {Journal of Graph Theory},
  year    = {2011},
  volume  = {67},
  number  = {3},
  pages   = {226--234},
  doi     = {10.1002/jgt.20527},
  url     = {https://doi.org/10.1002/jgt.20527}
}

@article{WCvB,
author = {Cambie, Stijn and Cames van Batenburg, Wouter and Davies, Ewan and Kang, Ross J.},
title = {Packing list-colorings},
journal = {Random Structures \& Algorithms},
volume = {64},
number = {1},
pages = {62-93},
%keywords = {graph colouring, graph packing, independent transversals, list colouring, strong chromatic number},
doi = {https://doi.org/10.1002/rsa.21181},
url = {https://onlinelibrary.wiley.com/doi/abs/10.1002/rsa.21181},
eprint = {https://onlinelibrary.wiley.com/doi/pdf/10.1002/rsa.21181},
%abstract = {Abstract List coloring is an influential and classic topic in graph theory. We initiate the study of a natural strengthening of this problem, where instead of one list-coloring, we seek many in parallel. Our explorations have uncovered a potentially rich seam of interesting problems spanning chromatic graph theory. Given a \$\$ k \$\$-list-assignment \$\$ L \$\$ of a graph \$\$ G \$\$, which is the assignment of a list \$\$ L(v) \$\$ of \$\$ k \$\$ colors to each vertex \$\$ v\in V(G) \$\$, we study the existence of \$\$ k \$\$ pairwise-disjoint proper colorings of \$\$ G \$\$ using colors from these lists. We may refer to this as a list-packing. Using a mix of combinatorial and probabilistic methods, we set out some basic upper bounds on the smallest \$\$ k \$\$ for which such a list-packing is always guaranteed, in terms of the number of vertices, the degeneracy, the maximum degree, or the (list) chromatic number of \$\$ G \$\$. (The reader might already find it interesting that such a minimal \$\$ k \$\$ is well defined.) We also pursue a more focused study of the case when \$\$ G \$\$ is a bipartite graph. Our results do not yet rule out the tantalising prospect that the minimal \$\$ k \$\$ above is not too much larger than the list chromatic number. Our study has taken inspiration from study of the strong chromatic number, and we also explore generalizations of the problem above in the same spirit.},
year = {2024}
}

\end{document}